\documentclass[a4paper, 12pt]{amsart}
\usepackage{amsmath, amsthm, amsfonts, amscd, amssymb, eucal, latexsym, mathrsfs, appendix, parskip, tikz, leftidx}
\usepackage[numbers,sort&compress]{natbib}
\usepackage{color,hyperref}
\usepackage{indentfirst}
\hypersetup{colorlinks,breaklinks,
            linkcolor=red,urlcolor=red,
            anchorcolor=red,citecolor=blue}

\theoremstyle{definition}
\numberwithin{equation}{section}

\newtheorem*{rep@theorem}{\rep@title}
\newcommand{\newreptheorem}[2]{%
\newenvironment{rep#1}[1]{%
 \def\rep@title{#2 \ref{##1}}%
 \begin{rep@theorem}}%
 {\end{rep@theorem}}}
\makeatother

\newreptheorem{theorem}{Theorem}

\newtheorem{theorem}{Theorem}[section]

\newtheorem{lemma}[theorem]{Lemma}
\newtheorem{proposition}[theorem]{Proposition}

\newtheorem{definition}[theorem]{Definition}
\newtheorem{remark}[theorem]{Remark}

\newtheorem{example}[theorem]{Example}

\newtheorem*{conjecture*}{Conjecture}

\allowdisplaybreaks

\newcommand{\fai}{\varphi}

\newcommand{\Tr}{{\rm Tr}}

\newcommand{\smnoind}{{\smallskip\noindent}}

\newcommand{\id}{{\rm id}}

\newcommand{\SG}{{\rm SG}}
\newcommand{\Pol}{{\rm Pol}}
\newcommand{\spec}{{\rm Spec}}

\newcommand{\abs}[1]{\left\vert#1\right\vert}
\newcommand{\norm}[1]{\left\|#1\right\|}

\newcommand{\CI}{\mathcal{I}}

\newcommand{\CA}{\mathcal{A}}
\newcommand{\CB}{\mathcal{B}}

\newcommand{\KH}{\mathfrak{H}}

\newcommand{\KK}{\mathfrak{K}}

\newcommand{\KM}{\mathfrak{M}}

\newcommand{\BR}{\mathbb{R}}
\newcommand{\BN}{\mathbb{N}}
\newcommand{\BZ}{\mathbb{Z}}

\newcommand{\BC}{\mathbb{C}}

\newcommand{\BG}{{\mathbb{G}}}

\newcommand{\Bo}{\mathbf{1}}

\newcommand{\BGd}{{\widehat{\mathbb{G}}}}

\newcommand{\Rep}{\mathrm{Rep}}

\newcommand{\supp}{\mathrm{supp}\,}

\newcommand{\Irr}{{\rm Irr}}

\begin{document}
\title[Amenable fusion algebraic actions of discrete quantum groups]{Amenable fusion algebraic actions of discrete quantum groups on compact quantum spaces}

\author{Xiao Chen}
\address{Xiao Chen, School of Mathematics and Statistics, Shandong University, Weihai, 264209, PR China}
\email{chenxiao@sdu.edu.cn}
\thanks{Xiao Chen is supported by the NSFC no. 11701327.}
\author{Debashish Goswami}
\address{Debashish Goswami, Statistics-Mathematics Unit, Indian Statistical Institute, Kolkata, India}
\email{goswamid@isical.ac.in}
\author{Huichi Huang}
\address{Huichi Huang, College of Mathematics and Statistics, Chongqing University, Chongqing, 401331, P.R. China}
\email{huanghuichi@cqu.edu.cn}
\thanks{Huichi Huang was partially supported by NSFC no. 11871119 and Chongqing Municipal Science and Technology Commission fund no. cstc2018jcyjAX0146.}

\begin{abstract}
In this paper, we introduce actions of fusion algebras on unital $C^*$-algebras, and define amenability for fusion algebraic actions.
Motivated by S.\ Neshveyev et al.'s work, considering the co-representation ring of a compact quantum group as a fusion algebra, we define the canonical fusion algebraic (for short, CFA) form of a discrete quantum group action on a compact quantum space.
Furthermore, through the CFA form, we define FA-amenability of discrete quantum group actions, and present some basic connections between FA-amenable  actions and amenable discrete quantum groups.
As an application, thinking of a state on a unital $C^*$-algebra as a ``probability measure" on a compact quantum space, we show that amenability for a discrete quantum group is equivalent to both of FA-amenability for an action of this discrete quantum group on a compact quantum space and the existence of this kind of ``probability measure" that is FA-invariant under this action.

\vspace{04pt}

\vspace{04pt}

\end{abstract}


\maketitle

\section{Introduction}

The notion of amenability essentially begins with Lebesgue (1904): could Lebesgue's measure be extended as a finitely additive measure defined on all subsets of $\BR^n$ which is invariant under any isometry?
Hausdorff answered negatively for $n\geqslant 3$ in 1914 and Banach positively for $n\leqslant 2$ in 1923.
The class of amenable groups was introduced and studied by von Neumann (1929) and used to explain why the Banach-Tarski Paradox occurs only for $\BR^n$ ($n\geqslant 3$), i.e., the desired measures mention above exist for $\BR^n$ ($n\leqslant 2$).
Nowadays, amenability has become an important concept in many fields, and has been generalized for various subjects, such as Banach algebras, $C^*$-algebras, $W^*$-algebras, operator spaces and quantum groups. For more details about the theory and history of amenability, the readers may refer to \cite{Pat1988}, \cite{Rund2002}, \cite{Green1969} and \cite{Bed-Tus2003}.

Since 1940s, the main concern shifts from finitely additive measures to means: integrating a positive, finitely additive probability measure on a space $\mathfrak{X}$ gives rise to a state on $L^\infty(\mathfrak{X})$. In \cite[Page 18]{Green1969}, the following famous problem was formulated:

\noindent
{\bf Q1:} if, under a certain group $G$ action on a topological space $\mathfrak{X}$, there exists a finitely additive $G$-invariant measure, also known as $G$-invariant mean, on $\mathfrak{X}$, is $G$ amenable?

The question remained longtime unanswered until van Douwen gave a counterexample in \cite{Dou1990}.
Viewing a von Neumann algebra as a quantum measure space,  the author in \cite[Proposition 3.6]{Anant1979} answered {\bf Q1} in the von Neumman algebraic setting, that is, a locally compact group $G$ is amenble if and only if there is an $G$-invariant mean under an amenable action of $G$ on a von Neumann algebra.
Recently, in \cite{Moa2018},  the above result \cite[Proposition 3.6]{Anant1979} had been generalized to the discrete quantum group case.
Since a $C^*$-algebra is thought of as a quantum topological space, it is then a natural question to propose a $C^*$-algebraic and quantum analogue of the question {\bf Q1} as follows.

\noindent
{\bf Q2:} if there exists an invariant state on a $C^*$-algebra under an action of a discrete quantum group, then is this quantum group amenable?

For the question {\bf Q2} above, we give a new realization of discrete quantum group actions on compact quantum spaces.  In \cite{Hab-Hat-Nesh2021} and \cite{Nes-Yam2014}, S.\ Neshveyev and his students shows that every discrete quantum group action on a unital $C^*$-algebra is equivalent to a module algebraic structure on this unital $C^*$-algebra. In Section~\ref{sec:CFA-act} of this paper, we use this module algebraic structure to define the \emph{canonical fusion algebraic form} (for short, CFA form, see Definition~\ref{def:dqg-cfa-act}) of a discrete quantum group action on a unital $C^*$-algebra, which means that a usual action of a discrete quantum group is actually a special class of fusion algebraic actions defined in Definition~\ref{def:fa-act}.

More interestingly, compared with Vaes's definition of discrete quantum group actions (\cite[Definition 1.24]{Vae-Ver2007}), the manner of a CFA action preserves the ``point by point act" feature of the ordinary discrete group action. So the CFA form of a discrete quantum group action, which is closer to the classical mode, looks like more direct and more visual.
That is why we here characterize discrete quantum group actions in fusion algebraic setting instead of in the widely used way.

In Section~\ref{sec:amen-CFA-act}, using the CFA form of a discrete quantum group action, we introduce the notion of a \emph{FA-amenability} action, which is compatible with Ozawa's definition (see \cite{Brow-Ozaw2008}), and also show in Proposition~\ref{prop:basicequiv} that a discrete quantum group is amenable if and only if the trivial action on $\BC$ is FA-amenable if and only if  every action is FA-amenable.
Finally, we apply FA-invariant states and FA-amenable actions to give a characterization of amenable discrete quantum groups in Theorem~\ref{thm:mainthm}:
\begin{quote}
A discrete quantum group is amenable, if and only if there exists an FA-invariant state under an FA-amenable action.
\end{quote}
This theorem, which provides a positive answer to the question {\bf Q2} above, can be considered as a $C^*$-algebraic analogue of \cite[Theorem 4.7]{Moa2018}, and in fact is also another quantum version of \cite[Propostion 3.6]{Anant1979}.

Before the main parts of this paper, as a preliminary, we will in the next section introduce some notations and recall some basic theories of compact quantum groups, discrete quantum groups and fusion algebras.

\bigskip

\section{Preliminaries and notations}\label{sec:pre-not}

\noindent
{\bf\emph{Conventions.}} Within this paper, we denote by $\BC$ (resp., $\BR$, $\BZ$, $\BN_0$)  the set of all complex numbers (resp., real numbers, integers and non-negative integers),  and by $B(\KH,\KK)$ the space of bounded linear operators from a complex Hilbert space $\KH$ to another space $\KK$, and $B(\KH)$ stands for $B(\KH,\KH)$. We use the convention that the inner product $\langle \cdot , \cdot \rangle$ of a complex Hilbert space $\KH$ is conjugate-linear in the second variable. Let $\id_X$  be the identity operator on a space or an algebra $X$.

The notation $\CA\otimes\CB$ always means the minimal tensor product of two $C^*$-algebras $\CA$ and $\CB$. For any integer $n\geqslant 1$, we denote by $M_n(\CA)$ be the $C^*$-algebra of all $n\times n$-matrices whose entries lie in a $C^*$-algebra $\CA$, by $I_n$ the unit matrix, and by $\Tr$ a normalized trace on $M_n(\CA)$ or $B(\KH)$.

Define ${\rm diag}[x;n]:=\begin{pmatrix}
  x &        & O   \\
      & \ddots &     \\
  O   &        & x \\
\end{pmatrix}_{n\times n}$ for any $n\in\BN$ and any element $x$ in some algebra.
The Kronecker symbol is denoted by $\delta_{i,j}:= \left\{\begin{array}{ll}
  1, & i=j\\
  0, & i\neq j
\end{array} \right.$, for any $i$ and $j$ in a index set $\CI$.

In this paper all $C^*$-algebras are assumed to be separable.

\subsection{Compact and discrete quantum groups}\label{subsec:cdqg}
Consider compact quantum groups in the sense of S.\ L.\ Woronowicz (see \cite{Wor1987} and \cite{Wor1998}), which are non-commutative analogues of compact groups. Every discrete quantum group is usually defined as the dual of a compact quantum group.  For more details about compact and discrete quantum groups, the reader may refer to  \cite {Fra-Ska-Sol2017}, \cite{Nes-Tus2013}, \cite{Tim2008} and references therein.

\begin{definition}\label{def:cqg}
A \textbf{compact quantum group} is a pair $(\CA, \Delta)$ consisting of a unital $C^*$-algebra $\CA$ with its unit $\Bo_\CA$, and a {\bf co-product} $\Delta$ which is a unital $*$-homomorphism  from $\CA$ to $\CA\otimes\CA$
such that
\begin{enumerate}
  \item $(\id\otimes\Delta)\Delta=(\Delta\otimes\id)\Delta$;
  \item $\Delta(\CA)(\Bo_\CA\otimes\CA)=\Delta(\CA)(\CA\otimes\Bo_\CA)$ are dense in $\CA\otimes\CA$.
\end{enumerate}
\end{definition}

One may think of $\CA$ in Definition~\ref{def:cqg} as $C(\BG)$, the $C^*$-algebra of continuous functions on a compact quantum space $\BG$ with a quantum group structure. In the rest of the paper we write a compact quantum group $(\CA,\Delta)$ with its unit $\Bo_\CA$ as $(C(\BG),\Delta_\BG)$ with its unit $\Bo_\BG$.

A {\bf non-degenerate (unitary) co-representation} $U$ of a compact quantum group $\BG$ is an invertible (unitary) element in $\KM(K(\KH_U)\otimes C(\BG))$ satisfying $U_{12}U_{13}=(\id\otimes\Delta)U$, where $K(\KH_U)$ is the $C^*$-algebra of compact linear operators on a complex Hilbert space $\KH_U$, and $\KM(K(\KH_U)\otimes C(\BG))$ is the multiplier  $C^*$-algebra of $K(\KH_U)\otimes C(\BG)$.
Here both of $U_{12}$ and $U_{13}$ are the usual ``leg notation", similarly hereinafter.
The Hilbert space $H_U$ is called {\bf carrier Hilbert space} of $U$.
From now on, we always assume co-representations are non-degenerate.

It is well-established that every finite dimensional co-representation is equivalent to a unitary co-representation, and every irreducible co-representation is finite dimensional. Let $\Rep(\BG)$ be the class of equivalent classes of all finite dimensional unitary co-representations of $\BG$.
For every $\gamma\in\Rep(\BG)$, let $U^\gamma\in\gamma$ be unitary representative of $\gamma$, and $\KH_\gamma$ be its carrier Hilbert space with dimension $d_\gamma$.

Upon choosing an orthonormal basis in $\KH_\alpha$ we can express $U^\alpha$ as a unitary matrix $(u^\alpha_{ij})^{d_\alpha}_{i,j=1}$ in $M_{d_\alpha}(C(\BG))$, that is,
$$U^\alpha=\sum^{d_\alpha}_{i,j=1}e_{ij}\otimes u^\alpha_{ij},$$
where $\{e_{ij}\}_{i,j=1,2,...,d_\alpha}$ is the corresponding basis of matrix units in $B(\KH_\alpha)\cong M_{d_\alpha}(\BC)$.

For any given two $\alpha, \beta\in\Rep(\BG)$ with their matrix forms $U^\alpha$ and $U^\beta$, define the {\bf direct sum}, denoted by $\alpha+\beta$, as the equivalent class of  unitary representations of dimension $d_\alpha+d_\beta$, given by $\begin{pmatrix}
  U^\alpha &   O   \\
  O   &    U^\beta \\
\end{pmatrix}$,
and the {\bf tensor product}, denoted by $\alpha\beta$, as the equivalent class of  unitary representations of dimension $d_\alpha d_\beta$, given by $U^{\alpha\beta}=U_{13}^{\alpha}U_{23}^{\beta}$.

Also for $\gamma\in\Rep(\BG)$, the {\bf conjugate} (also known as {\bf contragredient} or {\bf adjoint}) co-representation $\bar{\gamma}$ of $\gamma$ is given by $\overline{U_\gamma}:=((u^\gamma_{ij})^*)^{d_\gamma}_{i,j=1}$ on the conjugate Hilbert space $\overline{\KH}_\gamma$ of $\KH_\gamma$.

For any $\alpha\in\Rep(\BG)$, the {\bf character} $\chi(\alpha)$ of $\alpha$ is given by $$\chi(\alpha)=\sum^{d_\alpha}_{i=1}u^{\alpha}_{ii}.$$
Note that $\chi(\alpha)$ is independent of the choice of representatives of $\alpha$. Also we have $\norm{\chi(\alpha)}\leqslant d_\alpha$, since $\sum^{d_\alpha}_{k=1}u^\alpha_{ik}(u^\alpha_{ik})^*=\Bo_\BG$ for every $1\leqslant i\leqslant d_\alpha$. Moreover,
$$\chi(\alpha+\beta)=\chi(\alpha)+\chi(\beta),\quad \chi(\alpha\beta)=\chi(\alpha)\chi(\beta),\quad \chi(\alpha)^*=\chi(\bar\alpha)$$
for finite-dimensional co-representations $\alpha$ and $\beta$.

Every co-representation of a compact quantum group is a direct sum of irreducible representations.
Let $\BGd$ be the set of equivalent classes of irreducible unitary co-representations of $\BG$.
The algebra $\Pol(\BG)$ generated by all matrix coefficients arising from irreducible co-representations becomes a Hopf $*$-algebra (with the restricted co-product) which is dense in $C(\BG)$.

Denote by $h_{\BG}$ the unique Haar state on $\BG$ with GNS-construction $(L^2(\BG), \theta_{\BG}, \Lambda_{\BG})$.
Here $\theta_{\BG}$ is a $C^*-$algebra isomorphism from $C(\BG)$ to $\theta_{\BG}(C(\BG))$,
$\Lambda_{\BG}$ is a quotient map from $C(\BG)$ to $L^2(\BG)$, and
$$L^2(\BG)=\bigoplus_{\alpha\in\BGd} H^\alpha=\bigoplus_{\alpha\in\BGd} \KH_\alpha\otimes\KH_{\bar\alpha},$$
with $H^\alpha$ the subspace of $L^2(\BG)$ spanned by matrix elements $\{u^\alpha_{ij}\}^{d_\alpha}_{i,j=1}$ of $U^\alpha$.

In our paper, we identify a {\bf discrete quantum group} with $\BGd$ as the dual of some compact quantum group $\mathbb{G}$.
The $C^*$-algebra of the  virtual objects--``continuous functions" on $\BGd$ is a $c_0$-direct sum $$c_0(\BGd)=\bigoplus_{\alpha\in\BGd} B(\KH_\alpha)=\bigoplus_{\alpha\in\BGd} M_{d_\alpha}(\BC).$$

And the fundamental unitary of a compact quantum group $\BG$ can be written as
$$W_\BG=\bigoplus_{\alpha\in\Irr(\BG)} U^\alpha \in M(c_0(\BGd)\otimes C(\BG))\subseteq B(L^2(\BG)\otimes L^2(\BG)),$$
where $U^\alpha$ is a unitary representative of $\alpha$, and $\subseteq$ above makes sense under the GNS-construction $(L^2(\BG), \theta_{\BG}, \Lambda_{\BG})$.
Denote  $W_{\BGd}$ by the fundamental unitary of $\BGd$ given by $\Sigma  W^*_\BG\Sigma,$ where $\Sigma$ denotes the usual flip map.
Also we denote by $\Delta_{\BGd}$ the co-product of $\BGd$, and have the following well-known formulae
\begin{equation}\label{eq:delta-W-1}
(\Delta_{\BGd}\otimes\id)(W_\BG)=(W_\BG)_{23}(W_\BG)_{13},
\end{equation}
and
\begin{equation}\label{eq:delta-W-2}
(\varepsilon_\BGd\otimes\id)(W_\BG)=\Bo_\BG,
\end{equation}
where $\varepsilon_\BGd$ is the co-unit of $\BGd$.

Throughout this paper, we assume that every compact quantum group $\CA=C(\BG)$ is {\bf second countable}, i.e., $\BG$ has a separable underlying $C^*$-algebra, which amounts to saying $\BGd$ is at most countable and  the GNS space arising from the Haar state is separable.

\subsection{Fusion algebras}\label{sec:fa}
In this part we recall the notions of fusion algebras and amenability for fusion algebras. These topics were treated by F. Hiai and M. Izumi (\cite{Hia-Izu1998}). Other references on the subject are \cite{Kyed2008} and references therein.

\begin{definition}\label{def:fa}
A {\bf fusion algebra} $R(\CI)$ is a unital ring which is a free $\BZ$-module $\BZ[\CI]$ with a basis $\CI$ such that
\begin{enumerate}
  \item The unit $e$ is in $\CI$.
  \item The abelian monoid $\BN_0[\CI]$ is closed under multiplication, that is, for all $\alpha,\beta$ in $I$, there exists uniquely a family of nonnegative integers $(N_{\alpha,\beta}^\gamma)_{\gamma\in \CI}$ such that
      $$\alpha\beta=\sum_{\gamma\in \CI}N_{\alpha,\beta}^\gamma\gamma.$$
    \item There exists a $\mathbb{Z}$-linear, anti-multiplicative map~(called {\bf involution}) $x\to\bar{x}$ preserving $\CI$.
    \item {\bf Fronenius reciprocity} holds: $$N_{\alpha,\beta}^\gamma=N_{\gamma,\bar{\beta}}^\alpha=N_{\bar{\alpha},\gamma}^\beta$$ for all $\alpha,\beta,\gamma\in \CI$.
  \item There exists a $\BZ$-linear multiplicative map $d: \BZ[\CI]\to [1,\infty)$ such that $d(x)=d(\bar{x})$ for all $x\in \CI$. We call $d$ the {\bf dimension function}.
\end{enumerate}
\end{definition}

Note that the multiplicativity of $d$ implies
$$1=\sum_{\gamma\in\CI}\frac{d(\gamma)}{d(\alpha)d(\beta)}N_{\alpha,\beta}^\gamma,$$
for all $\alpha,\beta\in\CI$.
For an element $f=\sum_{\alpha\in\CI}k_\alpha\alpha\in R(\CI)$, the set $\{\alpha\in\CI\,\lvert\,k_\alpha\neq 0\}$ is called the {\bf support of $f$} and denoted $\supp f$.
We shall also consider the {\bf complexified fusion algebra} $\BC\otimes_{\BZ}\BZ[\CI]$ which will be denoted $\BC[\CI]$ in the following.
Note that this becomes a complex $*$-algebra with the induced algebraic structures, and the trace $\tau_\CI$ is extended to $\BC[\CI]$.

\begin{example}\label{ex:fa}
$(a)$ For any discrete group $\Gamma$, there exists a compact quantum group  $\BG$ which is the dual group of $\Gamma$. On one hand, we know that $C(\BG)$ is the reduced group $C^*$-algebra $C^*_r(\Gamma)$ of $\Gamma$.
On the other hand, the $C^*$-algebra $c_0(\Gamma)$, which is the $c_0$-direct sum $\bigoplus_{\gamma\in\Gamma}\BC$, can be presented as the $c_0$-direct sum $\bigoplus_{\alpha\in\BGd}M_{d_\alpha}(\BC)$ and $d_\alpha=1$. Namely, every element in $\Gamma$ can be viewed as a one-dimensional co-representation of $\BG$, and $\Gamma$ can be actually identified with $\BGd$.
In particular, if $\Gamma$ is abelian, then $\BG$ is the compact dual group $\widehat\Gamma$ consisting of all one-dimensional representations of $\Gamma$, and also $\BGd=\Gamma$.
Hence, the integral group ring $\BZ[\Gamma]$ becomes a fusion algebra with the basis $\Gamma$, when endowed with (the  $\BZ$-linear exension of) inversion as involution and trivial dimension function given by $d(\alpha)=d_\alpha=1$ for all $\alpha\in\Gamma$. The complexified fusion algebra $\BC[\Gamma]$ (i.e., $\Pol(\BG)$) is just the usual complex group ring of $\Gamma$ , whose completion, under the norm induced by left regular representation (resp., universal representation), is the reduced group $C^*$-algebra $C^*_r(\Gamma)$ (resp., the universal group $C^*$-algebra $C^*_u(\Gamma)$).

\smnoind
$(b)$ If $\BG=(C(\BG),\Delta_\BG)$ is a compact quantum group, its irreducible co-representations constitute the basis of a fusion algebra with tensor product as multiplication.
As explained in Section~\ref{subsec:cdqg}, for any two $\alpha,\beta\in\BGd$, the tensor product $\alpha\beta$ of $\alpha$ and $\beta$ can be de decomposed as
 $$\alpha\beta=\sum_{\gamma\in\BGd}N_{\alpha,\beta}^\gamma\gamma,$$
where $N_{\alpha,\beta}^\gamma=0$ for all but finitely many $\gamma$ in $\BGd$.
Also \cite[Proposition 3.4]{Wor1987} (or \cite[Example 2.3]{Kyed2008}) implies that he Frobenius reciprocity law holds for $\BGd$.
Thus, a product can be defined on the free $\BZ$-module $\BZ[\BGd]$ by the tensor product above, while the addition comes from the direct sum on $\BGd$.

Moreover, the trivial co-representation $e=1\otimes\Bo_\BG\in\BGd$ is a unit for this product.
Also we can extend the conjugate operation $\alpha\in\BGd\mapsto\bar\alpha\in\Rep(\BG)$ to the ring $\BZ[\BGd]$, and define a dimension function $d:\BZ[\BGd]\rightarrow [1,+\infty)$ by $d(\alpha)=d_\alpha$.

Hence, corresponding to a compact quantum group $\BG$,  we obtain a fusion algebra $R(\BGd)$, also viewed as a free $\BZ$-module $\BZ[\BGd]$, with its basis $\BGd$.
\end{example}

Next, we introduce the convolution product and the convolution operators associated with a fusion algebra $R(\CI)$. All terminologies and notations below mainly come from \cite{Hia-Izu1998} and \cite{Kyed2008}.

Denote by $\mu_0$ the counting measure with the weight $d^2$, that is, $\mu_0(\{\xi\})=d(\xi)^2$ for any $\xi\in\CI$.
For $1\leqslant p\leqslant +\infty$, let $\ell^p(\CI)$ (resp., $\ell^p(\CI,\mu_0)$) be the usual complex Banach $\ell^p$-space on $\CI$ with the counting measure (resp., the measure $\mu_0$ defined above), whose norm is denoted $\norm{\cdot}_p$ (resp., $\norm{\cdot}_{p,\mu_0}$).
The Dirac measure at $\xi\in\CI$ is denoted by $\delta_\xi$, which is also viewed as a function on $\CI$ such that it sends $\xi$ to $1$ and all other elements in $\CI$ to $0$.

All $\ell^p$-space have a common dense subspace denoted by $C_c(\CI)$, which consists of all typical element of  the form $f=\sum_{\xi\in\CI} c_\xi\delta_\xi$ with $\supp f$ being finite and $c_\xi\in\BC$. We also identify an element $\sum_{\xi\in\CI} c_\xi\delta_\xi\in C_c(\CI)$ with a finite support measure when $c_\xi\geqslant 0$.

For $\xi, \eta\in\CI$ we define the (weighted) convolution of the corresponding Dirac measure, $\delta_\xi$ and $\delta_\eta$, as
$$\delta_\xi*\delta_\eta=\sum_{\alpha\in\CI}\frac{d(\alpha)}{d(\xi)d(\eta)}N^\alpha_{\xi,\eta}\delta_\alpha\in\ell^1(\CI).$$

This extends linearly and continuously to a submultiplicative product on $\ell^1(\CI).$
For $f\in\ell^\infty(\CI)$ and $\xi\in\CI$ we define $\lambda_\xi(f), \rho_\xi(f):\CI\rightarrow\BC$ by
$$\lambda_\xi(f)(\eta)=\sum_{\alpha\in\CI}f(\alpha)(\delta_{\bar\xi}*\delta_\eta)(\alpha)$$ and
$$\rho_\xi(f)(\eta)=\sum_{\alpha\in\CI}f(\alpha)(\delta_\eta*\delta_\xi)(\alpha).$$

For each $f\in\ell^\infty(\CI)$ we have  $\lambda_\xi(f), \rho_\xi(f)\in\ell^\infty(\CI)$ and for each $p\in\BN\cup\{\infty\}$
the map $\lambda_\xi, \rho_\xi:\ell^\infty(\CI)\rightarrow\ell^\infty(\CI)$ restrict to bounded operators on $\ell^p(\CI)$ denoted $\lambda_{p,\xi}$ and $\rho_{p,\eta}$ respectively. By linear extension, we therefore obtain a map
$\lambda_{p,-}:\BZ[\CI]\rightarrow B(\ell^p(\CI,\sigma))$, which respects the weighted convolution product. Similarly, this map also can be extended onto the space of finitely supported measures on $\CI$ (resp., $\ell^p(\CI)$), denoted by $\lambda_{p,\mu}$ for any finitely support measure $\mu$ on $\CI$ (resp., $\lambda_{p,\phi}$ for any $\ell^p$-function $\phi\in\ell^p(\CI)$).
For more details, the reader may refer to \cite[Page 675-676]{Hia-Izu1998} and \cite[Section 2]{Kyed2008}.

\begin{remark}
In particular when $p=2$, if the operator $U_{\CI,\mu_0}:\ell^2(\CI)\rightarrow\ell^2(\CI,\mu_0)$ is defined by $U_{\CI,\mu_0}(\delta_\eta)=\frac{1}{d(\eta)}\delta_\eta$ for any $\eta\in\CI$, then $U_{\CI,\mu_0}$ is unitary and intertwines with the operator
$$l_\xi:\delta_\eta\longmapsto\frac{1}{d(\xi)}\sum_{\alpha\in\CI}N^\alpha_{\xi,\eta}\delta_\alpha, \quad\forall \xi,\,\eta\in\CI.$$
See \cite[Proposition 2.5]{Kyed2008} and \cite[Remark 1.4]{Hia-Izu1998}\label{rem:intw-l-leftrepn}.
\end{remark}

A fusion algebra $R(\CI)=\BZ[\CI]$ is called {\bf finitely generated} if there exists a finitely supported probability measure $\mu$ on $\CI$ such that $\CI=\bigcup_{n\in\BN}\supp\mu^{*n}$ and $\mu(\bar\xi)=\mu(\xi)$ for all $\xi\in\CI$.
Here the first condition is referred to as {\bf non-degeneracy} of $\mu$ and the second condition is referred to as {\bf symmetry} of $\mu$.

Note that for any compact quantum group $\BG$ its co-representation fusion algebra $R(\BGd)$ is finitely generated exactly.  In this paper,  we choose the following definition of amenability for fusion algebras due to  Kyed, which is motivated by  \cite[Definition 4.3]{Hia-Izu1998} that is defined for finitely generated fusion algebras.

\begin{definition} \cite[Definition 3.1]{Kyed2008}\label{def:amen-fa}
A fusion algebra $R(\CI)=\BZ[\CI]$ is called {\bf amenable} if $1\in\spec(\lambda_{2,\mu})$ for every finitely supported, symmetric probability measure $\mu$ on $\CI$. Here $\spec(\lambda_{2,\mu})$ denotes the spectrum of the operator $\lambda_{2,\mu}$ on the Hilbert space $\ell^2(\CI,\mu_0)$.
\end{definition}

\begin{definition} \cite[Definition 3.1]{Hia-Izu1998}\label{def:almost-inv-seq}
For any positive integer $p$, a sequence $\{f_n\}_{n\in\BN_0}$ in $\ell^p(R(\CI),\mu_0)$ is called {\bf an almost invariant $\ell^p$-sequence} if $\norm{f_n}_{p,\mu_0}=1$, $n\in\BN_0$, and
$$\lim_{n\rightarrow+\infty}\norm{\lambda_{p,\gamma}(f_n)-f_n}_{p,\mu_0}=0, \quad \forall \gamma\in\CI.$$
\end{definition}

\begin{remark}\label{rem:amen-fa-izum-kyed}
From \cite[Corollary 4.4]{Hia-Izu1998}, it follows that, for any positive integer $p$,
if $R(\CI)$ has an almost invariant $\ell^p$-sequence, then $1\in\spec(\lambda_{2,\mu})$ for every (not necessarily non-degenerate, nor symmetric) probability measure $\mu$ on $\BGd$, and hence $R(\CI)$ is amenable by Definition~\ref{def:amen-fa}.
In other words, Izumi's amenability for fusion algebras in \cite{Hia-Izu1998} can implies Kyed's amenability in \cite{Kyed2008}.
\end{remark}

We end this section with the following remarkable theorem, which reveals the equivalence between amenability for a discrete quantum group and amenability for its fusion algebra. In fact, it is worthy mentioning that, the equivalence between amenability of a discrete quantum group and amenability of the classical dimension function in the sense of Hiai and Izumi had been known long before \cite{Kyed2008}, and the key idea goes back to T.\ Banica and G.\ Skandalis (see \cite[Section 6]{Banica1999}).

\begin{theorem}\cite[Theorem 4.5]{Kyed2008}\label{thm:amen-fa-dqg}
A discrete quantum group $\BGd$ is amenable if and only if the co-representation ring $R(\BGd)$ is amenable.
\end{theorem}

Since we will primarily be interested in the co-representation rings of  second countable compact quantum groups, the bases of fusion algebras will be assumed to be countable throughout the paper.

\bigskip

\section{Canonical fusion algebraic actions of discrete quantum groups on compact quantum spaces}\label{sec:CFA-act}

From the viewpoint of noncommuative topology, we usually think of a unital $C^*$-algebra $\CA$ as continuous functions on a {\bf compact quantum space} $\mathfrak{X}$, that is, $\CA=C(\mathfrak{X})$.

In this section, we will define actions of fusion algebras on compact quantum spaces, and further characterize the actions of discrete quantum groups on compact quantum spaces in the fusion algebraic setting. Namely, we will consider a discrete quantum group action as a fusion algebraic action, since every discrete quantum group uniquely corresponds to a fusion algebra (see Example~\ref{ex:fa}(b)).

\begin{definition}\label{def:fa-act}
Let $\CA$ be a unital $C^*$-algebra, $R(\CI)$ be a fusion algebra with a basis $\CI$, and $\SG_{\CI}$ be the semigroup generated by $\CI$ with respect to ring multiplication, i.e., $\SG_{\CI}=\{\gamma_1\gamma_2\cdots\gamma_k\ \lvert\ \gamma_1,..., \gamma_k\in\CI, k\in\BN\}$.
An \textbf{action of $R(\CI)$ on $\CA$} is a map $(\alpha, a)\mapsto \sigma_\alpha(a)$ from $\SG_{\CI}\times \CA$ to $\CA$ such that
\begin{enumerate}
  \item For any fixed $\alpha$ in $\SG_{\CI}$, the map $\sigma_\alpha: \CA\rightarrow \CA,\ a\mapsto\sigma_\alpha(a)$ is $\BC$-linear, unital, norm-contractive and preserves $*-$operation.
  \item $\sigma_\alpha\circ\sigma_\beta=\sigma_{\alpha\beta}$ for any $\alpha, \beta$ in $\SG_{\CI}$.
    \item $\sigma_e$ is an identity map on $\CA$.
\end{enumerate}
Simply, we denote by $\sigma: R(\CI)\curvearrowright \CA$ the fusion algebraic action above.
\end{definition}

Nowadays, the usual definition, presented below, of actions of compact and discrete quantum groups on $C^*$-algebras  mainly origins from the theory of Hopf $C^*$-algebraic actions, which was introduced in \cite{Baaj-Skan1993}. For this topic, the readers also may refer to \cite{Vae-Ver2007}.

\begin{definition} \cite[Definition 1.24]{Vae-Ver2007}\label{def:usual-act}
A (left) action of a discrete quantum group $\widehat{\mathbb{G}}$ on a unital $C^*$-algebra $\CA$ is a non-degenerate $*$-homomorphism $\rho: \CA\to \mathfrak{M}(c_0(\widehat{\mathbb{G}})\otimes \CA)$ such that
\begin{enumerate}
  \item $(\id\otimes\rho)\circ\rho=(\Delta_{\widehat{\mathbb{G}}}\otimes\id)\circ\rho$;
  \item $(\widehat{\varepsilon}\otimes \id)\circ\rho=\id$.
\end{enumerate}
Here we denote this action by $\rho:\BGd \curvearrowright \CA$.

If, for all $a\in\CA$, it holds that $\rho(a)=\Bo_{\BGd}\otimes a$ where $\Bo_\BGd$ is the unit in $\mathfrak{M}(c_0(\BGd))$, we call this $\rho$ a {\bf trivial action}.
\end{definition}

In \cite[Section 1.2]{Hab-Hat-Nesh2021} and \cite[Proposition 1.5]{Nes-Yam2014}, the authors show that a discrete quantum group action on a unital $C^*$-algebra is equivalent to a module algebra.

More precisely, given an action $\rho$ of a discrete quantum group $\BGd$ on a unital algebra $\CA$, we get a $\Pol(\BG)$-module structure on $\CA$ defined by
\begin{equation}\label{eq:modalg}
x\cdot a:=(x\otimes\id)\rho(a),\ \forall x\in\Pol(\BG),\  a\in\CA.
\end{equation}
This module structure is compatible with involution $*$ in the sense that
\begin{equation}\label{eq:prevstar}
x\cdot a^*=(S_\BG(x)^*\cdot a)^*,\ \forall x\in\Pol(\BG),\  a\in\CA,
\end{equation}
where $S_\BG$ is the antipode of the compact quantum group $\BG$.
Then $\CA$ become a $\Pol(\BG)$-module algebra, i.e., the module structure respects the algebra structure on $\CA$, that is,
$$x\cdot (ab)=(x_{(1)}\cdot a)(x_{(2)}\cdot b), \ \forall x\in\Pol(\BG) \text{ and } a,b\in\CA,$$
where we use Sweedler's sumless notation, so we write $\Delta_\BG(x)=x_{(1)}\otimes x_{(2)}$.

Conversely,  since $c_0(\BGd)$ is $c_0$-direct sum $\bigoplus_{\alpha\in\BGd} B(\KH_\alpha) (=\bigoplus_{\alpha\in\BGd} M_{d_\alpha}(\BC))$, we know that, if a $C^*$-algebra $\CA$ is a $\Pol(\BG)$-module algebra satisfying \eqref{eq:prevstar}, then there is a non-degenerate $*$-homomorphism $\rho: \CA\rightarrow \mathfrak{M}(c_0(\widehat{\mathbb{G}})\otimes \CA)$ defined by
\begin{equation}\label{equ:rho}
\rho(a):=\bigoplus_{\alpha\in \widehat{\mathbb{G}}}\alpha(a),\ \forall a\in \CA,
\end{equation}
where the map
\begin{equation}\label{eq:map-alpha}
\alpha(a):=(u^\alpha_{ij}\cdot a))_{1\leqslant i,j\leqslant d_\alpha}=\sum_{i,j=1}^{d_\alpha}e_{i,j}\otimes (u_{i,j}\cdot a) =U^\alpha\cdot(I_{d_\alpha}\otimes a)
\end{equation}
which is denoted by $\beta_{U_\alpha}$ in \cite[Proposition 1.5]{Nes-Yam2014}, is a non-degenerate $*$-homomorphism from $\CA$ to $M_{d_\alpha}(\CA)(=M_{d_\alpha}(\mathbb{C})\otimes \CA)$. This $\rho$ is just the unique action, induced by this module algebra (i.e., satisfies \eqref{eq:modalg}), of $\BGd$ on $\CA$ in the sense of Definition~\ref{def:usual-act}.

Also, we can check that, this module algebra mentioned above can be equivalently formulated as follows.
\begin{definition}\label{def:cq-spc}
A unital $C^*$-algebra $\CA$ is called \textbf{a $\Pol(\BG)$-module algebra} for a compact quantum group $\BG$, if every element $x$ in $\Pol(\mathbb{G})$ is a norm-bounded unital $*$-linear map on $\CA$, denoted by $a\mapsto x\cdot a\ (\forall a\in\CA)$, such that
\begin{enumerate}
\item for every $\alpha$ in $\widehat{\mathbb{G}}$, the map $\alpha(a):=(u^\alpha_{ij}\cdot a))_{1\leqslant i,j\leqslant d_\alpha}(=U^\alpha\cdot(I_{d_\alpha}\otimes a))$ with $a$ in $\CA$ is a non-degenerate $*$-homomorphism from $\CA$ to $M_{d_\alpha}(\CA)(=M_{d_\alpha}(\mathbb{C})\otimes \CA)$;
\item $u^\alpha_{ij}\cdot(u^\beta_{kl}\cdot a)=(u^\alpha_{ij}u^\beta_{kl})\cdot a$ and $\Bo_\BG\cdot a=a$, for every $\alpha$, $\beta$ in $\widehat{\mathbb{G}}$, $1\leqslant i,\ j,\ k,\ l\leqslant d_\alpha$;

\end{enumerate}
\end{definition}

The following lemma is a direct consequence from Definition~\ref{def:cq-spc}.
\begin{lemma}\label{lem:dyn-prop}
Let $\CA$ be a $\Pol(\BG)$-module algebra as above. Then one has $$\|\chi(\alpha)\cdot a\|\leqslant d_\alpha\|a\| \text{ by Definition~\ref{def:cq-spc}(1)},$$ and
$$\chi(\alpha\beta)\cdot a=\chi(\alpha)\cdot(\chi(\beta)\cdot a) \text{ by Definition~\ref{def:cq-spc}(2)},$$
for any $a\in \CA$.
\end{lemma}

The next example tells us that a discrete quantum group action is a special case of fusion algebraic actions.

\begin{example}\label{ex:fa-act}
Suppose that $\CA$ is a $\Pol(\BG)$-module algebra for a compact quantum group $\BG$. Considering the fusion algebra $R(\BGd)$ with the basis $\BGd$ as in Example~\ref{ex:fa}(b), by Definition~\ref{def:cq-spc} and Lemma~\ref{lem:dyn-prop}, we can directly check that, for any $\gamma\in \SG_\BGd$, the following map on $\CA$ given by
\begin{equation}\label{eq:dqg-cfa-act}
\sigma_\gamma:a\rightarrow \frac{\chi(\gamma)}{d_\gamma}\cdot a
\end{equation}
satisfies all conditions in Definition~\ref{def:fa-act}. Hence we can see that every $\Pol(\BG)$-module algebra $\CA$ for a compact quantum group $\BG$ gives a fusion algebraic action of $R(\BGd)$ on $\CA$.
\end{example}

However, so far as I know, it can not be asserted now that all fusion algebraic action of $R(\BGd)$ on $\CA$ are induced by a $\Pol(\BG)$-module algebra. In fact, we guess it is not true! So, from Example~\ref{ex:fa-act},  we maybe find a broader category of quantum dynamic systems on unital $C^*$-algebras, that is, fusion algebraic dynamic system, which contains all discrete quantum group actions in the sense of Definition~\ref{def:usual-act}.
In other words, if we use the fusion algebra $R(\BGd)$ to stand in for a discrete quantum group $\BGd$, then Definition~\ref{def:usual-act} is probably not  the unique fashion of an action of $\BGd$ on a unital $C^*$-algebra. Hence, Example~\ref{ex:fa-act} actually gives a a canonical fusion algebraic form of a discrete quantum group action in the usual sense.

\begin{definition}\label{def:dqg-cfa-act}
Let $\CA$ be a  $\Pol(\mathbb{G})$-module algebra for a compact quantum group $\BG$, and $\rho:\BGd \curvearrowright \CA$ be the unique action corresponding to this module structure.
Then, the action of the fusion algebra $R(\BGd)$ on $\CA$ in Example~ \ref{ex:fa-act} is said to be \textbf{the canonical fusion algebraic form} (for short, {\bf CFA form}) of $\rho$ on $\CA$, or \textbf{a canonical fusion algebraic action} (for short, {\bf CFA action}) of $\BGd$ on $\CA$.

Here we denote this action by $\sigma: \BGd \overset{\textrm{\tiny{cfa}}}{\curvearrowright} \CA$.
\end{definition}

\begin{example}\label{ex:dqg-cfa-act}
$(a)$ Let $\BG$ be a compact quantum group. Given an injective $*$-representation $\pi: C(\mathbb{G})\to B(\KH)$, one can define an action of $C(\mathbb{G})$ on $B(\KH)$ by $$u_{ij}^\alpha\cdot T=\sum_{l=1}^{d_\alpha}\pi(u_{il}^\alpha) T\pi(u_{jl}^\alpha)^*$$ for  every $T$ in $B(H)$, every $\alpha$ in $\widehat{\mathbb{G}}$ and all $1\leq i,j\leq d_\alpha$. Here $\Pol(\BG)$-value unitary matrix $U^\alpha=(u_{ij}^\alpha)^{d_\alpha}_{i,j=1}$ is a unitary representative of $\alpha$, and $d_\alpha$ is the dimension of $\alpha$.

Define $\pi(U^\alpha):=(\pi(u_{ij}^\alpha))^{d_\alpha}_{i,j=1}$.
One can check that
$$\alpha(T):=\pi(U^\alpha){\rm diag}[T;d_\alpha]\pi(U^\alpha)^*$$ is an injective $*$-homomorphism from $B(\KH)$ to $M_{d_\alpha}(B(\KH))$, and
$$\chi(\alpha)\cdot T=\sum_{i,j=1}^{d_\alpha} \pi(u_{ij}^\alpha) T \pi(u_{ij}^\alpha)^*$$ is a positive map with operator norm not greater than $d_\alpha$.
Hence, by Definition~\ref{def:cq-spc}, it can be easily verified that $B(\KH)$ is a $\Pol(\BG)$-module algebra.
Consequently, the CFA action of $\BGd$ on $B(\KH)$ is given by
$$\sigma_{\alpha}(T):=\frac{\chi(\alpha)}{d_\alpha}\cdot T,$$
for any $\alpha\in\BGd$ and any $T\in B(\KH)$.

\smnoind
$(b)$ Let $\mathfrak{X}$ be a compact Hausdorff space equipped with an action of a discrete group $\Gamma$. Then we can define an action of $\Gamma$ on the continuous function algebra $C(\mathfrak{X})$ by $(\gamma\cdot f)(x):=f(\gamma^{-1}\cdot x)$ for any $\gamma\in\Gamma$, $x\in\mathfrak{X}$ and $f\in C(\mathfrak{X})$. From Example~\ref{ex:fa}(a),  it is apparent that $C(\mathfrak{X})$, as a unital $C^*$-algebra, is a $\BC[\Gamma]$-module algebra, and the action defined above is also a CFA action of $\Gamma$ (as a discrete quantum group) on $C(\mathfrak{X})$.
In other words, the concept of the CFA action is a direct generalization of the classical discrete group action on a compact space.

\smnoind
$(c)$ For any compact quantum group $\BG$, the \textbf{trivial $\Pol(\BG)$-module algebra} is a unital $C^*$-algebra $\CA$ with setting $u_{ij}^\alpha\cdot a:=\delta_{i,j}a$, which implies that $\alpha(a)=I_{d_\alpha}\otimes a$, for every $\alpha\in\BGd$ and any $a\in\CA$. Consequently, the corresponding {\bf trivial (canonical fusion algebraic) action} of $\BGd$ on $\CA$ is given by
$$\sigma_{\alpha}(a):=\frac{\chi(\alpha)}{d_\alpha}\cdot a=a,$$
for any $\alpha\in\BGd$ and any $a\in \CA$. Clearly, this trivial action is just the usual trivial action in the Vaes's definition above.

\end{example}

For a discrete quantum group action $\rho:\BGd \curvearrowright \CA$ in Definition~\ref{def:usual-act}, we usually define the invariant state under this action $\rho$ in the following way.

\begin{definition}\label{def:dqg-invstat1}
Let $\CA$ be a unital $C^*$-algebra, and $\BG$ be a compact quantum group.
For any given a discrete quantum group action $\rho:\BGd \curvearrowright \CA$, a state $\fai$ on $\CA$ is called \textbf{invariant} under $\rho$ (or {\bf $\rho$-invariant}) if $\fai[(\omega\otimes\id)\rho(a))]=\omega(\Bo_\BGd)\fai(a)$ for  every $a\in \CA$ and $\omega\in c_0(\BGd)^*_+$. Here $\Bo_\BGd$ is the unit in $\mathfrak{M}(c_0(\BGd))$, and $c_0(\BGd)^*_+$ is the set of all norm one positive linear funcitionals on $c_0(\BGd)$.
\end{definition}

Similarly, for fusion algebraic actions, we also can define invariant states.

\begin{definition}\label{def:dqg-invstat}
For any given $\sigma: R(\CI)\curvearrowright \CA$ be an action of a fusion algebra $R(\CI)$ on a unital $C^*$-algebra $\CA$, a state $\fai$ on $\CA$ is called \textbf{invariant} under $\sigma$ (or {\bf $\sigma$-invariant}) if $\fai(\sigma_\alpha(a))=\fai(a)$ for every $\alpha$ in the basis $\CI$ and $a$ in $\CA$.

In particular, when $\CI$ is a discrete quantum group $\BGd$,
For any action $\rho:\BGd \curvearrowright \CA$ with its CFA form $\sigma:\BGd \overset{\textrm{\tiny{cfa}}}{\curvearrowright} \CA$, a state $\fai$ on $\CA$ is called \textbf{FA-invariant} under $\rho$ (or {\bf $\rho$-FA-invariant}) if $\fai(\chi(\alpha)\cdot a)=d_\alpha \fai(a)$ i.e., $\fai(\sigma_\alpha(a))=\fai(a)$, for every $\alpha$ in $\widehat{\mathbb{G}}$ and $a$ in $\CA$.
\end{definition}

The next proposition tells us that the FA-invariance is weaker than the usual invariance.

\begin{proposition}\label{prop:relation-invstate}
Let $\BGd$ be  a discrete quantum group, and $\CA$ be a unital $C^*$-algebra.
For any given action $\rho:\BGd \curvearrowright \CA$ with its CFA form $\sigma:\BGd \overset{\textrm{\tiny{cfa}}}{\curvearrowright} \CA$, every invariant state must be an FA-invariant state.
\end{proposition}

\begin{proof}
Considering $\CA$ as a $\Pol$-module algebra induced by $\rho:\BGd \curvearrowright \CA$, so, by \eqref{equ:rho},  we can write $\rho$ as
$$\rho(a)=\bigoplus_{\alpha\in \widehat{\mathbb{G}}}\sum^{d_\alpha}_{i,j=1}e_{ij}\otimes u^\alpha_{ij}\cdot a,$$ for any $a\in \CA$.  Let $\fai$ be a $\rho$-invariant state on $\CA$.
Then the invariance of $\fai$ in the sense of Definition~\ref{def:dqg-invstat1}, if and only if, for every $\omega\in c_0(\BGd)^*_+$ and any $a\in \CA$, one has
$$\bigoplus_{\alpha\in \widehat{\mathbb{G}}}\sum^{d_\alpha}_{i,j=1}\omega(e_{ij})\fai(u^\alpha_{ij}\cdot a)=\bigoplus_{\alpha\in \widehat{\mathbb{G}}}\omega(I_{d_\alpha})\fai(a),$$
By the orthogonality of the copies $\{\KH_\alpha\}_{\alpha\in\Irr(\BG)}$, the above equation is equivalent to
\begin{equation}\label{eq:inv1}
\sum^{d_\alpha}_{i,j=1}\omega(e_{ij})\fai(u^\alpha_{ij}\cdot a)=\omega(I_{d_\alpha})\fai(a),
\end{equation}
for any $\alpha\in \widehat{\mathbb{G}}$.
Since the equation~(\ref{eq:inv1}) holds for every $\omega\in c_0(\BGd)^*_+$,
we may choose a special $\omega_0$ defined by $\omega_0(T)=\sum^{d_\alpha}_{i=1}\langle T(e_i), e_i\rangle$ for any $T\in B(\KH_\alpha)$.
Because $\omega_0(e_{ij})=\delta_{i,j}$,
we can obtain $$\fai(\chi(\alpha)\cdot a)=d_\alpha \fai(a),\text{ that is, }\fai(\sigma_\alpha(a))=\fai(a),$$
for every $\alpha\in\BGd$ and every $a\in \CA$.
\end{proof}

\begin{remark}\label{rem:relation-invstate}
From the proof of the above proposition, we can see that, the $\rho$-invariance of a state $\fai$ on $\CA $ actually implies, in the sense of weak topology, that $\fai(\alpha(a))=I_{d_\alpha}\otimes \fai(a)$ for every $\alpha\in\BGd$ and every $a\in\CA$, while the $\rho$-FA-invariance is equivalent to the fact that
$\Tr(\fai(\alpha(a)))=\Tr(I_{d_\alpha}\otimes\fai(a))$ for any $\alpha\in\BGd$ and any $a\in\CA$. Clearly, the former is stronger that the latter.
\end{remark}

\bigskip

\section{FA-amenable actions, FA-invariant states and amenable discrete quantum groups}\label{sec:amen-CFA-act}

In this section, we will be devoted to discussing the existence of FA-invariant states and characterizing amenability for discrete quantum groups in a dynamical way.

\smnoind
{\bf\emph{Notation:}} the weighted cardinality of a finite set $F$ of  $\widehat{\mathbb{G}}$, denoted by $|F|_w$, is given by
$\sum_{\alpha\in F} d_\alpha^2$.

\begin{definition}\cite[Corollary 4.10]{Kyed2008}\label{def:amen-cqg}
A discrete quantum group $\widehat{\mathbb{G}}$ is \textbf{amenable} if there exists a sequence $\{F_n\}$ of finite subsets of $\widehat{\mathbb{G}}$ such that
  $$\lim_{n\to\infty}\frac{|\partial_\gamma F_n|_w}{|F_n|_w}=0$$ for every $\gamma$ in $\widehat{\mathbb{G}}$, where
  \begin{align*}
    \partial_\gamma F_n=&\{\alpha\in F_n\,\lvert\, \text{there exists some}\, \beta\notin F_n\, \text{with}\, N_{\alpha,\gamma}^\beta>0\} \\
    &\cup\{\alpha\notin F_n\,\lvert\, \text{there exists some}\, \beta\in F_n\, \text{with}\, N_{\alpha,\gamma}^\beta>0\}.
  \end{align*}
Here $\{F_n\}_{n\in\BN_0}$ is called {\bf a sequence of F$\o$lner sets}.
\end{definition}

According to both \cite[Theorem 3.1]{Bed-Tus2003} and \cite[Theorem 4.3]{Kyed2008}, a (locally) compact quantum group is called {\bf co-amenable} if its reduced group $C^*$-algebra is canonically isomorphic to its universal group $C^*$-algebra. This definition also appears in some more earlier references, such as \cite[Definition 6.1]{Banica1999}.
In \cite[Corollary 4.10]{Kyed2008}, Definition~\ref{def:amen-cqg} was actually used to equivalently describe co-amenability of compact quantum groups.
For a discrete quantum group, amenability is equivalent to co-amenability of its compact dual.
Moreover, in present paper, we consider the co-representation ring of a compact quantum group as its dual object, and view every element in this ring as a``non-commutative discrete group element". Then the F$\o$lner sets in Definition~\ref{def:amen-cqg} just lie in its dual.

\begin{proposition}\label{prop:invstate-exist}
If an amenable discrete quantum  group $\BGd$ acts on a compact quantum space $\CA$, then there always exists a FA-invariant state on $\CA$ under this action.
\end{proposition}
\begin{proof}
  Take an arbitrary state $\fai$ on $\CA$. Define
 $\fai_n$ by $$\fai_n(a)=\frac{1}{|F_n|_w}\sum_{\alpha\in F_n} d_\alpha\fai(\chi(\alpha)\cdot a)$$ for every $a$ in $\CA$.

Since $\|\chi(\alpha)\|\leq d_\alpha$, we have $\|\fai_n\|\leq 1$. Note that $\fai_n(1_\CA)=1$. Hence $\fai_n$ is a state on $\CA$. Equipped with weak-$*$ topology, there exists a limit point of $\fai_n$, say, $\psi$.

  Next we prove that $\psi$ is an invariant state.

  For every $a$ in $\CA$ and $\beta$ in $\widehat{\mathbb{G}}$, we have
  $$\fai_n(\chi(\beta)\cdot a)-d_\beta\fai_n(a)
    =\frac{1}{|F_n|_w}\sum_{\alpha\in F_n} d_\alpha\fai(\chi(\alpha)\chi(\beta)\cdot a)-\frac{1}{|F_n|_w}\sum_{\alpha\in F_n} d_\alpha d_\beta\fai(\chi(\alpha)\cdot a). $$
 On one hand
 \begin{align*}
              &\frac{1}{|F_n|_w}\sum_{\alpha\in F_n} d_\alpha\fai(\chi(\alpha)\chi(\beta)\cdot a)
               =\frac{1}{|F_n|_w}\sum_{\alpha\in F_n}d_\alpha\sum_{\gamma\in\widehat{\mathbb{G}}}N_{\alpha,\beta}^\gamma\fai(\chi(\gamma)\cdot a) \\
               &=\frac{1}{|F_n|_w}\sum_{\alpha\in F_n}d_\alpha\left(\sum_{\gamma\in F_n}+\sum_{\gamma\notin F_n}\right) N_{\alpha,\beta}^\gamma\fai(\chi(\gamma)\cdot a)
 \end{align*}
 On the other hand
 \begin{align*}
  &\frac{1}{|F_n|_w}\sum_{\alpha\in F_n} d_\alpha d_\beta\fai(\chi(\alpha)\cdot a)  =\frac{1}{|F_n|_w}\sum_{\alpha\in F_n} d_\alpha d_{\bar{\beta}}\fai(\chi(\alpha)\cdot a) \\
  &=\frac{1}{|F_n|_w}\sum_{\alpha\in F_n} \sum_{\gamma\in\widehat{\mathbb{G}}}N_{\alpha,\bar{\beta}}^\gamma d_\gamma\fai(\chi(\alpha)\cdot a)=\frac{1}{|F_n|_w}\sum_{\alpha\in F_n} \sum_{\gamma\in\widehat{\mathbb{G}}}N_{\gamma,\beta}^\alpha d_\gamma\fai(\chi(\alpha)\cdot a)  \\
  &=\frac{1}{|F_n|_w}\sum_{\gamma\in F_n} \sum_{\alpha\in\widehat{\mathbb{G}}}N_{\alpha,\beta}^\gamma d_\alpha\fai(\chi(\gamma)\cdot a)=\frac{1}{|F_n|_w}\sum_{\gamma\in F_n} \left(\sum_{\alpha\in F_n}+\sum_{\alpha\notin F_n}\right)N_{\alpha,\beta}^\gamma d_\alpha\fai(\chi(\gamma)\cdot a).
 \end{align*}
 So
  \begin{align*}
    &|\fai_n(\chi(\beta)\cdot a)-d_\beta\fai_n(a)| \\
    &=\abs{\frac{1}{|F_n|_w}\sum_{\alpha\in F_n}d_\alpha\sum_{\gamma\notin F_n} N_{\alpha,\beta}^\gamma\fai(\chi(\gamma)\cdot a)-\frac{1}{|F_n|_w}\sum_{\gamma\in F_n}\sum_{\alpha\notin F_n}N_{\alpha,\beta}^\gamma d_\alpha\fai(\chi(\gamma)\cdot a)} \\
   &\leqslant \frac{1}{|F_n|_w}\sum_{\alpha\in F_n}\sum_{\gamma\notin F_n}d_\alpha N_{\alpha,\beta}^\gamma d_\gamma\|a\|+\frac{1}{|F_n|_w}\sum_{\gamma\in F_n}\sum_{\alpha\notin F_n}N_{\alpha,\beta}^\gamma d_\alpha d_\gamma\|a\|  \\
   &\leqslant d_\beta \frac{|\partial_\beta F_n|_w}{|F_n|_w}\|a\|\to 0,
  \end{align*}
  as $n\to\infty$.

  This proves that $\psi$ is invariant.
\end{proof}

Let $\CA$ be a unital $C^*$-algebra, and $R(\CI)$ be a fusion algebra with a basis $\CI$.

For any two elements $f=\sum_{\alpha\in\CI}\delta_\alpha\otimes a_\alpha,\ g=\sum_{\beta\in\CI}\delta_\beta\otimes b_\beta$ in $C_c(\CI)\otimes\CA$, we put a $\CA$-value inner product on $C_c(\CI)\otimes\CA$ by
$$\langle f\,\lvert\, g\rangle_\CA:=\sum_{\alpha,\beta\in\CI}\delta_{\alpha,\beta} a_\beta b^*_\beta\in\CA,$$
a $\CA$-right (left) multiplier on $C_c(\CI)\otimes\CA$ by
$$fc:=\sum_{\alpha\in\CI}\delta_\alpha\otimes a_\alpha c \ \left(cf:=\sum_{\alpha\in\CI}\delta_\alpha\otimes ca_\alpha\right),\ \forall c\in\CA$$
and a norm
$$\norm{f}_{2,\CA}:=\norm{\langle f\,\lvert\, f\rangle_\CA}_\CA^{\frac{1}{2}}.$$
Under this norm, we can see that the completion of $C_c(\CI)\otimes\CA$ is $\ell^2(\CI)\otimes\CA$.

For the given $f$ and $g$ above, we define a twisted convolution related to a fusion algebraic action $\sigma:R(\CI)\curvearrowright \CA$ by
$$f*_\sigma g:=\sum_{\alpha,\beta\in\BGd}l_{\alpha}(\delta_\beta)\otimes a_\alpha\sigma_\alpha(b_\beta).$$

When $\CI$ is a discrete quantum group $\BGd$, applying the orthogonality relations in $\Pol(\BG)$ (see \cite[Theorem 1.4.3]{Nes-Tus2013} or \cite[Section 5]{Wor1987}), the inner product defined above coincides with the canonical product on $C_c(\BGd)$ determined by
\begin{equation}\label{eq:innerproduct}
\langle \delta_\alpha\,\lvert\,\delta_\beta\rangle:=h_\BG(\chi(\alpha)\chi(\beta)^*)=\delta_{\alpha,\beta}, \quad \forall\alpha,\beta\in\BGd.
\end{equation}

\begin{definition}\label{def:amen-act}
An action $\sigma: R(\CI)\curvearrowright \CA$ of a fusion algebra $R(\CI)$ on a unital $C^*$-algebra $\CA$ is said to be {\bf amenable}, if there exist a sequence $\{\xi_n\}_{n\in\BN_0}\in C_c(\CI)\otimes\CA^+$ such that
\begin{enumerate}
  \item $\langle\xi_n\,\lvert\,\xi_n\rangle_\CA=1_\CA$.
  \item $\xi_na=a\xi_n$ for any $a\in\CA$.
  \item $\norm{\delta_\gamma*_\sigma\xi_n-\xi_n}_{2,\CA}\rightarrow 0$ as $n\rightarrow +\infty$, for any $\gamma\in\BGd$.
\end{enumerate}
The sequence $\{\xi_n\}_{n\in\BN_0}$ is called {\bf a  F$\o$lner sequence for the action $\sigma$}.

In particular, when $\CI$ is a discrete quantum group $\BGd$, an action $\rho:\BGd\curvearrowright \CA$ is said to be {\bf FA-amenable} if the CFA form $\sigma:\BGd \overset{\textrm{\tiny{cfa}}}{\curvearrowright} \CA$ of $\rho$, as a fusion algebraic action, is amenable in the above sense.
\end{definition}

\begin{remark}\label{rem:amen-ozawa}
In the classical case that $\BGd$ is a discrete group $\Gamma$, our definition of FA-amenability for actions agrees with Ozawa's definition \cite[Definition 4.3.1]{Brow-Ozaw2008}. In particular, when $\BGd$ is a discrete group and $\mathfrak{X}$ is a compact space, the two definitions both coincide with the definition of amenability of a discrete group action on a compact space (see \cite[Definition 4.3.5 and Lemma 4.3.7]{Brow-Ozaw2008}).
\end{remark}

Firstly, we will need the following simple lemma.

\begin{lemma}\label{lem:appunit}
Suppose that $\BGd$ is amenable, and $\{F_n\}_{n\in\BN}$ is a sequence of F$\o$lner sets of $\BGd$.
Let $\xi_n=\dfrac{1}{\sqrt{|F_n|_w}}\sum_{\alpha\in F_n}d_\alpha \delta_\alpha\otimes 1_\CA$ for all $n\in\BN_0$.
For any $\alpha'\in\widehat{\mathbb{G}}$,  then one has  $$\lim_{n\to\infty}\norm{(l_{\alpha'}\otimes \id_\CA)(\xi_n)-\xi_n}_{2,\CA}=0.$$
\end{lemma}
\begin{proof}
Note that
\begin{align*}
(l_{\alpha'}\otimes \id_\CA)(\xi_n)&=\dfrac{1}{\sqrt{|F_n|_w}}\sum_{\alpha\in F_n}(l_{\alpha'}\otimes \id_\CA)(d_\alpha \delta_\alpha\otimes 1_\CA) \\
  &=\dfrac{1}{\sqrt{|F_n|_w}}\sum_{\alpha\in F_n}\left(\frac{d_\alpha}{d_{\alpha'}}\sum_{\beta\in\widehat{\mathbb{G}},\,  N^\beta_{\alpha',\alpha}>0}N^\beta_{\alpha',\alpha}\delta_\beta\otimes 1_\CA\right) \\
  &=\dfrac{1}{\sqrt{|F_n|_w}}\sum_{\alpha\in F_n}\frac{d_\alpha}{d_{\alpha'}}\left(\sum_{\beta\in F_n}+\sum_{\beta\notin F_n, \, N^\beta_{\alpha',\alpha}>0}\right) N^\beta_{\alpha',\alpha}\delta_\beta\otimes 1_\CA.
\end{align*}
Also
\begin{align*}
&\norm{\dfrac{1}{\sqrt{|F_n|_w}}\sum_{\alpha\in F_n}\frac{d_\alpha}{d_{\alpha'}}\sum_{\beta\notin F_n, \, N^\beta_{\alpha',\alpha}>0} N^\beta_{\alpha',\alpha}\delta_\beta\otimes 1_\CA}_{2,\CA}^2 \\
  &\leqslant\dfrac{1}{|F_n|_w}\sum_{\alpha\in F_n}\sum_{\beta\notin F_n, \, N^\beta_{\alpha',\alpha}>0}\frac{d_\alpha^2}{d_{\alpha'}^2} (N^\beta_{\alpha',\alpha})^2 \\
  &=\dfrac{1}{|F_n|_w}\sum_{\alpha\in F_n}\sum_{\beta\notin F_n, \, N^\beta_{\alpha',\alpha}>0}\frac{d_\alpha^2}{d_{\alpha'}^2} (N^\alpha_{\bar{\alpha'},\beta})^2 \\
  &\leqslant\dfrac{1}{|F_n|_w}\sum_{\beta\in\partial_{\bar{\alpha'}}F_n}d_\beta^2\to 0
\end{align*}
as $n\to +\infty$.
Moreover
\begin{align*}
&\dfrac{1}{\sqrt{|F_n|_w}}\sum_{\alpha\in F_n}\frac{d_\alpha}{d_{\alpha'}}\sum_{\beta\in F_n} N^\beta_{\alpha',\alpha}\delta_\beta\otimes 1_\CA \\
&=\dfrac{1}{\sqrt{|F_n|_w}}\sum_{\beta\in F_n}\sum_{\alpha\in F_n}\frac{d_\alpha}{d_{\alpha'}}N^\alpha_{\bar{\alpha'},\beta}\delta_\beta\otimes 1_\CA  \\
&=\dfrac{1}{\sqrt{|F_n|_w}}\sum_{\beta\in F_n}\left(\sum_{\alpha\in \widehat{\mathbb{G}}}-\sum_{\alpha\notin F_n, N^\alpha_{\bar{\alpha'},\beta}>0}\right)\frac{d_\alpha}{d_{\alpha'}}N^\alpha_{\bar{\alpha'},\beta}\delta_\beta\otimes 1_\CA.
\end{align*}
On one hand, one has
$$\dfrac{1}{\sqrt{|F_n|_w}}\sum_{\beta\in F_n}\sum_{\alpha\in \widehat{\mathbb{G}}}\frac{d_\alpha}{d_{\alpha'}}N^\alpha_{\bar{\alpha'},\beta}\delta_\beta\otimes 1_\CA=\dfrac{1}{\sqrt{|F_n|_w}}\sum_{\beta\in F_n}d_\beta\delta_\beta\otimes 1_\CA=\xi_n.$$
On the other hand
\begin{align*}
&\norm{\dfrac{1}{\sqrt{|F_n|_w}}\sum_{\beta\in F_n}\sum_{\alpha\notin F_n, N^\alpha_{\bar{\alpha'},\beta}>0}\frac{d_\alpha}{d_{\alpha'}}N^\alpha_{\bar{\alpha'},\beta}\delta_\beta\otimes 1_\CA}_{2,\CA}^2 \\
  &=\dfrac{1}{|F_n|_w}\sum_{\beta\in F_n}\sum_{\alpha\notin F_n, N^\alpha_{\bar{\alpha'},\beta}>0}\frac{d_\alpha^2}{d_{\alpha'}^2} (N^\beta_{\alpha',\alpha})^2 \\
   &\leqslant\dfrac{1}{|F_n|_w}\sum_{\beta\in\partial_{\bar{\alpha'}}F_n}d_\beta^2\to 0
\end{align*}
as $n\to\infty$.

We complete the proof.
\end{proof}

In the proposition below, we will apply the previous lemma to clarify the equivalence between FA-amenable actions and amenable discrete quantum groups, and generalize the classical result \cite[Exercise 4.4.1]{Brow-Ozaw2008}.

\begin{proposition}\label{prop:basicequiv}
Let $\BGd$ be a discrete quantum group. Then the following statements are equivalent:

\smnoind
$(a)$ The discrete quantum group $\BGd$ is amenable.

\smnoind
$(b)$ Every action is FA-amenable.

\smnoind
$(c)$  The trivial action on $\BC$ is FA-amenable.
\end{proposition}

\begin{proof}
$(a)\Rightarrow (b).$ Let $\CA$ be a unital $C^*$-algebra. For any given action $\rho:\BGd \curvearrowright \CA$ with its CFA form $\sigma:\BGd \overset{\textrm{\tiny{cfa}}}{\curvearrowright} \CA$, if $\BGd$ is amenable, then, using $\{\xi_n\}_{n\in\BN}$ in  Lemma~\ref{lem:appunit} to check the conditions in Definition~\ref{def:amen-act} one by one, we can immediately see that $\rho$ is FA-amenable.

\noindent
$(b)\Rightarrow (c).$ It is trivial.

\noindent
$(c)\Rightarrow (a).$ Let $\sigma_0:\BGd \overset{\textrm{\tiny{cfa}}}{\curvearrowright} \BC$ is a trivial FA-amenable CFA action of a discrete quantum group $\BGd$ on $\BC$.
Since $\sigma_0$ is amenable, there exists a F$\o$lner sequence $\{\xi_n\}_{n\in\BN_0}\in C_c(\BGd)$ such that
\begin{equation}\label{eq:approx-convolution}
\norm{\delta_\gamma*_{\sigma_0}\xi_n-\xi_n}_2\rightarrow 0
\end{equation}
as $n\rightarrow +\infty$ for any $\gamma\in\BGd$.

Due to the triviality of the action $\sigma_0$, we can see that $\delta_\gamma*_{\sigma_0} f=l_{\gamma}(f)$ for any $f\in\ell^2(\BGd)$ and every $\gamma\in\BGd$.  So, by \eqref{eq:approx-convolution}, we have
\begin{equation}
\lim_{n\rightarrow +\infty}\norm{l_\gamma(\xi_n)-\xi_n}_2\rightarrow 0 \text{ as }n\rightarrow +\infty, \quad\forall\gamma\in\BGd,
\end{equation}
which, combined with Remark~\ref{rem:intw-l-leftrepn},  ensures that
\begin{align}\label{eq:Fseq-leftrepn}
\lim_{n\rightarrow +\infty}\norm{\lambda_{2,\gamma}U_{\BGd,\mu_0}(\xi_n)-U_{\BGd,\mu_0}(\xi_n)}_{2,\mu_0}&=\lim_{n\rightarrow +\infty}\norm{U_{\BGd,\mu_0}^{-1}\lambda_{2,\gamma}U_{\BGd,\mu_0}(\xi_n)-\xi_n}_2\\
\nonumber &=\lim_{n\rightarrow +\infty}\norm{l_\gamma(\xi_n)-\xi_n}_2\rightarrow 0,\quad n\rightarrow +\infty.
\end{align}

For any $n\in\BN_0$, set $\zeta_n:=U_{\BGd,\mu_0}(\xi_n)$, whose $\norm{\cdot}_2$-norm is also $1$.

Combined with Remark~\ref{rem:amen-fa-izum-kyed}, the equation \eqref{eq:Fseq-leftrepn} tells us that the number $1$ is contained in $\spec(\lambda_{2, \mu})$ for every probability measure $\mu$ on $\BGd$. So, by Definition~\ref{def:amen-fa}, we know the fusion algebra $R(\BGd)$ is amenable, and so the discrete quantum group $\BGd$ is amenable because of Theorem~\ref{thm:amen-fa-dqg}.
\end{proof}

Finally, as an application of fusion algebraic actions, the following theorem, which is the main theorem in our paper, will provides a characterization of amenability for discrete quantum groups through FA-amenable actions and FA-invariant states on compact quantum spaces.

\begin{theorem}\label{thm:mainthm}
Let $\rho:\BGd \curvearrowright \CA$ be an action of a discrete quantum group $\BGd$ on a compact quantum space $\CA$, and $\sigma:\BGd \overset{\textrm{\tiny{cfa}}}{\curvearrowright} \CA$ be the CFA form of $\rho$. Then the following are equivalent:

\smnoind
$(a)$ The discrete quantum group $\BGd$ is amenable.

\smnoind
$(b)$ The action $\rho$ is FA-amenable, and there exists a $\rho$-FA-invariant state on the $C^*$-algebra $\CA$.
\end{theorem}

\begin{proof}
The implication $(a)\Rightarrow (b)$ directly follows from Proposition~\ref{prop:invstate-exist} and Proposition~\ref{prop:basicequiv}(b). It suffices to prove $(b)\Rightarrow (a)$.

Suppose that $\fai$ is a $\rho$-FA-invariant state on $\CA$. Since the action $\rho$ is FA-amenable, there exists a sequence $\{\xi_n\}_{n\in\BN_0}\subseteq C_c(\BGd)\otimes\CA^+$ such that
\begin{equation}\label{eq:approx-convolution-sigma}
\norm{\delta_\gamma*_\sigma\xi_n-\xi_n}_{2,\CA}\rightarrow 0 \text{ as } n\rightarrow +\infty,
\end{equation}
for any $\gamma\in\BGd$.

For any $n\in\BN_0$, we may assume that $\xi_n=\sum_{\alpha\in \BGd}\delta_\alpha\otimes a_{\alpha,n}$, where  $a_{\alpha,n}\in\CA^+$.

Set $\eta_n:=(\id_{\ell^2(\BGd)}\otimes\fai)(\xi_n)=\sum_{\alpha\in \BGd}\fai(a_{\alpha,n})\delta_\alpha,$ and  let $\sigma_0:\BGd \overset{\textrm{\tiny{cfa}}}{\curvearrowright} \BC$ be a trivial action of $\BGd$ on $\BC$.

Then, by \eqref{eq:approx-convolution-sigma} and $\sigma$-CFA-invariance of $\fai$, one has
\begin{align*}
\norm{\delta_\gamma*_{\sigma_0}\eta_n-\eta_n}_2 &=\norm{\sum_{\alpha\in \BGd}l_\gamma(\delta_\alpha)\otimes\fai(a_{\alpha,n})-\eta_n}_2\\
&=\norm{\sum_{\alpha\in \BGd} l_\gamma(\delta_\alpha)\otimes\fai(\sigma_\alpha(a_{\alpha,n}))-\eta_n}_2\\
&=\norm{(\id_{\ell^2(\BGd)}\otimes\fai)(\delta_\gamma*_\sigma\xi_n)-(\id_{\ell^2(\BGd)}\otimes\fai)(\xi_n)}_2\\
&\leqslant\norm{\delta_\gamma*_\sigma\xi_n-\xi_n}_{2,\CA}\rightarrow 0, \text{ as }n\rightarrow +\infty,
\end{align*}
which implies that the trivial action $\sigma_0$ on $\BC$ is FA-amenable.

Hence $\BGd$ is amenable by Proposition~\ref{prop:basicequiv}(c).
\end{proof}

\begin{remark}
When $\mathfrak{X}$ is a compact space equipped with an action of a discrete group $\Gamma$, due to Example~\ref{ex:dqg-cfa-act}(b) and Remark~\ref{rem:amen-ozawa}, Theorem~\ref{thm:mainthm} actually implies that $\Gamma$ is amenable if and only if this action is amenable in the sense of \cite[Definition 4.3.5]{Brow-Ozaw2008} and there exists an invariant probability measure under this action.
\end{remark}

\bigskip

\section*{Acknowledgement}

We would like to sincerely appreciate S.\ Neshveyev for pointing out the references \cite{Hab-Hat-Nesh2021}\cite{Nes-Yam2014}\cite{Banica1999} to us, and giving us many constructive suggestions which improve our paper a lot.

\end{document}